\documentclass[10pt,reqno]{amsart}
\usepackage{graphicx} 

\usepackage{amsmath,amssymb,amsthm,mathtools}
\usepackage{geometry}
\usepackage{xcolor}
\usepackage{hyperref}
\hypersetup{
    colorlinks,
    linkcolor={red!50!black},
    citecolor={blue!50!black},
    urlcolor={blue!80!black}
}
\usepackage{enumitem, float}
\usepackage{bm}
\usepackage{xfrac, nccmath}
\usepackage{caption}
\usepackage{parskip}
\newcommand{\R}{\mathbb{R}}

\newcommand{\vv}{m}
\newcommand{\ttt}{t}
\newcommand{\TM}{\mathcal{T}}
\newcommand{\SM}{\mathcal{S}}
\newcommand{\G}{\mathcal{G}}
\newcommand{\tK}{\tilde{K}}
\newcommand{\K}{K}
\newcommand{\compset}{\mathcal{A}_T}

\newtheorem{theorem}{Theorem}[section]
\newtheorem{lemma}[theorem]{Lemma}

\newtheorem{corollary}[theorem]{Corollary}

\theoremstyle{definition}

\numberwithin{equation}{section}

\title{Qualitative stability for a family of trace Sobolev inequalities}
	\author[R. Neumayer]{Robin Neumayer}
	\address{Department of Mathematical Sciences, Carnegie Mellon University, 5000 Forbes Avenue, Pittsburgh, PA 15213, United States of America }
	\email{neumayer@cmu.edu}

\begin{document}

\maketitle
\begin{abstract}
The goal of this short note is to prove qualitative stability for a family of trace Sobolev inequalities first proven by Carlen \& Loss for $p=2$ and by Maggi and the author for $p\in (1,n)$. This answers an open problem raised in a recent paper of Fan, Li \& Zhang and, in conjunction with their local analysis, yields sharp quantitative stability for this family of inequalities when $p=2$.
\end{abstract}
\section{Introduction}

Fix $n\geq 2$,  $p \in (1,n)$, and a half space $H = \{x \in \R^n : x\cdot e_1>0\}$. In \cite{MN}, Maggi \& Neumayer used a mass transportation argument to establish a one-parameter family of trace Sobolev inequalities on $H$, which encode the classical Sobolev and  Escobar inequalities as special cases. More specifically, consider the variational problem 
\begin{equation}\label{e: phi T}
\Phi_H(T) = \inf\{ \|\nabla u\|_{L^p(H)} :  u \in \compset\}\, \qquad T\geq 0,
\end{equation}
where the competitor class $\compset$ is given by
\begin{equation}
\compset= \{u \in \dot{W}^{1,p}(H) : \| u \|_{L^{p^*}(H)}=1, \| u\|_{L^{p^\sharp}(\partial H)} =T \} \,.
\end{equation}
Here the critical Sobolev exponents $p^* =\frac{np}{n-p}$ and $p^\sharp = \frac{(n-1)p}{n-p}$ are determined by scaling. Minimizers of $\Phi_H(T)$ were characterized in \cite{MN} for each $T>0$: the family of minimizers $\mathcal{M}_T$ comprises dilations and horizontal translations, and multiples by $\pm1$ of an explicit profile $U_T$ (recalled in section~\ref{ssec: extremals}). Thus,  in the resulting sharp Sobolev trace inequality
\begin{equation}\label{eqn: trace Sobolev}
\| \nabla u \|_{L^p(H)} \geq \Phi_H(T) \qquad \text{ for all } u \in \compset,
\end{equation}
equality holds if and only if $u \in \mathcal{M}_T.$  The case $T=0$ of \eqref{eqn: trace Sobolev} encodes the classical Sobolev inequality $\|\nabla u\|_{L^p(\R^n)}\geq S_{n,p}\|u\|_{L^{p^*}(\R^n)}$ for $u \in \dot{W}^{1,p}(\R^n)$, whose optimal constant $\Phi_H(0) = S_{n,p}$ and extremals on $\R^n$ were characterized in \cite{aubin1976, talenti1976best}. 
The Escobar inequality $\| \nabla u\|_{L^p(H)} \geq E_{n,p} \| u\|_{L^{p^\sharp}(\partial H)}$ for $u \in \dot{W}^{1,p}(H)$, whose sharp constant $E_{n,p}$ and extremals were given in \cite{Escobar1988,Nazaret2006} (see also \cite{Beckner93}), implies the linear lower bound $\Phi_H(T) \geq E_{n,p}T$ for all $T\geq 0$. This lower bound is saturated for exactly one value $T_E>0$ depending on $n$ and $p$.

When $p=2$, Carlen \& Loss \cite{carlenloss} first characterized minimizers of \eqref{e: phi T} using their method of competing symmetries \cite{CLcompetingsym}. In this case, the Sobolev and Escobar inequalities are closely linked to the Yamabe problem \cite{OGYamabe, Trudinger, AubinYamabe, SchYamabe, Escobar92, Escobar92b}.

With extremals of \eqref{eqn: trace Sobolev} characterized, stability is the  natural next question: if  $u \in \compset$ almost achieves equality in \eqref{eqn: trace Sobolev}, then is $u$ close, in a suitable sense, to some $v \in \mathcal{M}_T$?  Closeness to equality is quantified by the deficit
\[
\delta_T(u)= \|\nabla u\|_{L^p(H)}^p - \Phi_H(T)^p\,,
\]
while the strongest distance of a function $u \in \compset$ to the nearest extremal that one expects to control  is
\[
d_T(u) := \inf_{v \in \mathcal{M}_T}\| \nabla (u-v)\|_{L^p(H)}.
\]
In the recent paper \cite{FLZ}, Fan, Li \& Zhang gave a local quantitative analysis of the problem in the case $p=2$, showing that there exists $\alpha_T>0$ such that $\delta_T(u) \geq \alpha_T d_T(u)^2 + o(d_T(u)^2)$. The obstruction to turning this estimate into a global quantitative stability theorem was the absence of a {\it qualitative} stability result for \eqref{eqn: trace Sobolev}.  Here we fill this gap by establishing such a qualitative stability result. As a byproduct, we  resolve the open problem of global quantitative stability for this problem raised in \cite[Remark 1.1]{FLZ}. 
\begin{theorem}\label{thm: qual stability}
    Fix $p \in (1,n)$ and $T>0$. Given $\{u_k\} \subset \compset$, if $\delta_T(u_k) \to 0$, then $d_T(u_k)\to 0.$
\end{theorem}
Combining Theorem~\ref{thm: qual stability} with \cite[Theorem 1.1]{FLZ} yields global quantitative stability when $p=2$.
\begin{corollary}\label{cor: quant stability}
    Fix $p=2$ and $T>0$. There exists $\alpha'_T>0$ such that $\delta_T(u) \geq \alpha_T' d_T(u)^2$ for all $u \in \compset.$
\end{corollary}
Quantitative stability for the Escobar inequality  with $p=2$ was shown in \cite{Ho22}. For the Sobolev inequality on $\R^n$, sharp quantitative stability was shown in \cite{BianchiEgnell91} for $p=2$, see also \cite{DEFFL}, and for $p\in (1,n)$ in \cite{FZ}, after \cite{ciafusmag07, FN, Neumayer}.

\begin{figure}
    \centering
    \includegraphics[width=0.6\linewidth]{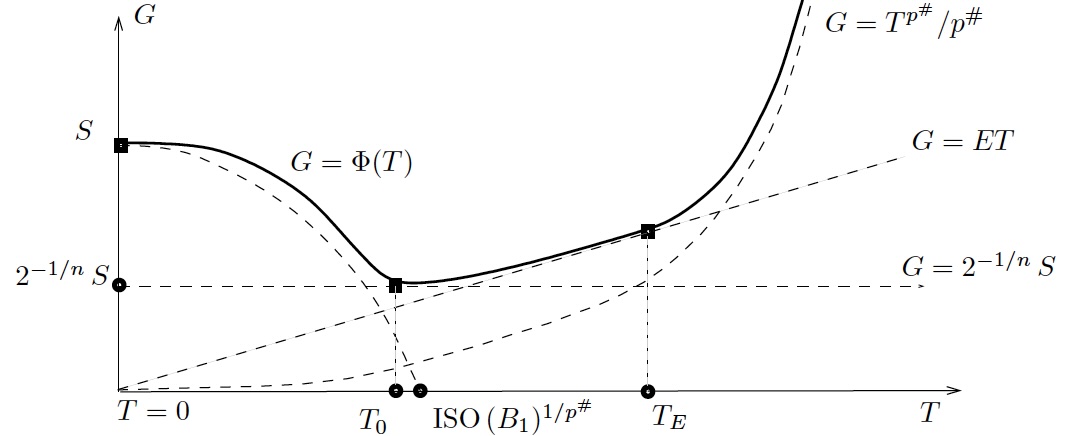}
    \caption{\small{A plot illustrating known properties of $T\mapsto \Phi_H(T)$; see section~\ref{ssec: extremals} for discussion.}}
    \label{fig1}
\end{figure}

At first glance, one might expect Theorem~\ref{thm: qual stability} to follow directly from standard concentration compactness and scaling methods \cite{lions1985}, as is the case for many functional inequalities, including the classical Sobolev inequality on $\R^n$ and the Escobar inequality on $H$. The fundamental difference here is that the sharp constant $\Phi_H(T)$ depends nontrivially on $T$. This means that splitting of mass cannot be ruled out through homogeneity and concave scaling.

More explicitly, concentration compactness arguments essentially reduce the proof of Theorem~\ref{thm: qual stability} to ruling out the possibility that a sequence  $\{u_k\}\subset \compset$ with $\delta_T(u_k)\to0$ fails to have $d_T(u_k)\to 0$  because it splits into two asymptotically non-interacting profiles $u_k^1$ and $u_k^2$. Suppose this happens, and $u_k^1$ and $u_k^2$ have $L^{p^*}(H)$ norms $\vv_1, \vv_2>0$ and $L^{p^\sharp}(\partial H)$ norms $\ttt_1,\ttt_2\geq 0$  respectively, with  $\vv_1^{p^*}+\vv_2^{p^*}=1$ and $\ttt_1^{p^\sharp}+\ttt_2^{p^\sharp}=T^{p^\sharp}$. 

For the classical Sobolev inequality, for instance, scaling easily shows such splitting is energetically too expensive: applying the Sobolev inequality to $u_k^1$ and $u_k^2$ separately shows that the total energy is at least $S_{n,p}^p (\vv_1^{p} + \vv_2^{p})$, which by strict concavity of $s\mapsto s^{p/p^*}$ is {\it strictly} larger than the infimal energy $S_{n,p}^p =S_{n,p}^p (\vv_1^{p^*}+\vv_2^{p^*})^{p/p^*} $. 

Instead, ruling out splitting in the present setting requires comparing the energy lower bound obtained by applying \eqref{eqn: trace Sobolev} to each profile separately, namely $\vv_1^p \Phi_H(\sfrac{\ttt_1}{\vv_1})^p +\vv_2^p \Phi_H(\sfrac{\ttt_2}{\vv_2})^p$, with the infimal energy $\Phi_H(T)^p$. These quantities cannot be related by scaling and have no clear strict ordering a priori. Even if $T$ is restricted to a particular interval, the ratios $\sfrac{\ttt_i}{\vv_i}$ may take any value in $[0,\infty)$.

In view of this discussion, the main tool to prove Theorem~\ref{thm: qual stability} is the following strict binding inequality for $\Phi_H$.
\begin{theorem}\label{thm: strict energy gap}
    Fix $T>0$ and let $\vv_1,\vv_2>0$ and $\ttt_1, \ttt_2\geq 0$ satisfy $\vv_1^{p^*} + \vv_2^{p^*} =1$ and $\ttt_1^{p^\sharp} + \ttt_2^{p^\sharp} =T^{p^\sharp}$. Then 
    \begin{equation}\label{eqn: strict convexity}
        \Phi_H(T)^p < \vv_1^p \Phi_H\left(\mfrac{\ttt_1}{\vv_1}\right)^p +\vv_2^p \Phi_H\left(\mfrac{\ttt_2}{\vv_2}\right)^p\,.
    \end{equation}
\end{theorem}

Another plausible approach to proving Theorem~\ref{thm: qual stability} would be to trace through the mass transportation argument of \eqref{eqn: trace Sobolev} from \cite{MN} (recalled in section~\ref{ssec:OT proof}) to extract information about the optimal transport map $\TM$ from $u^{p^*}\mathcal{L}^n$ to $U_T^{p^*}\mathcal{L}^n$ and use this to estimate $d_T(u)$. This has been done quantitatively for the isoperimetric inequality \cite{FMP}, the $1$-Sobolev inequality \cite{figmagpraa}, and the Sobolev inequality with $p\in(1,n)$ restricted to radially symmetric functions \cite{ciafusmag07}. This approach faces serious difficulties in the present context, and already for the Sobolev inequality on $\R^n$, in part because the control on $\TM$ degenerates in regions where  $u$ is small.

To prove Theorem~\ref{thm: strict energy gap}, we {\it do} use the mass transportation proof of \cite{MN}. The key difference is that we only need to control the optimal transport map $\TM$ for one explicit test function $w \in \compset$. We take $w$ to be the sum of cut off and translated copies of $\vv_1 U_{T_1}$ and $\vv_2 U_{T_2}$ where $T_i={\ttt_i}/{\vv_i}$. The energy $\int_H |\nabla w|^p$ can be  made arbitrarily close to the right-hand side of \eqref{eqn: strict convexity}. (Non-strict inequality in \eqref{eqn: strict convexity}  is immediate from testing against $w$.)

Say $\vv_1 \leq \vv_2$. The cyclical monotonicity of the graph of the optimal transport map $\TM$ associated with $w$ forces $\TM$ to map most of the mass of $\vv_1 U_{T_1}$ into a half-space $\{y_n >0\}$. On the other hand, a basic quantitative estimate obtained from the mass transportation proof (see \eqref{eqn: deficit and CS}) shows that if $\delta_T(w)$ is small, then $\TM-se_1$ is parallel to $-\nabla w$ on a set with a definite amount of mass for a fixed $s=s_T$. These properties are incompatible, forcing a definite lower bound for $\delta_T(w)$ and thus proving Theorem~\ref{thm: strict energy gap}.

\noindent{\bf Acknowledgements.} The author thanks Dejan Slep\v{c}ev for a useful discussion and Song Fan for valuable feedback on an earlier version of this manuscript.  
This work was partly supported by NSF CAREER grant DMS-2340195 and NSF RTG grant DMS-2342349. 

\section{Preliminaries}
In this section, we give some background and notation that will be needed in the rest of the paper. Let $n\geq2$ and $p \in (1,n)$ be fixed throughout. We let $\mathcal{L}^n$ be the Lebesgue measure on $\R^n$ and $\mathcal{H}^{n-1}$ the $(n-1)$-dimensional Hausdorff measure. 
Let $\dot{W}^{1,p}(H)$ be the space of $L^1_{\rm loc}(H)$ functions with distributional gradient in $L^p(H;\R^n)$ that vanish at infinity in the sense that $\mathcal{L}^n(\{|u|>t\})<\infty$ for every $t>0$. For a function $u \in \dot{W}^{1,p}(H)$ we simply write $u$ to refer to the trace $Tu \in L^{p^\sharp}(\partial H)$ of $u$.

\subsection{Minimizers for $\Phi_H(T)$}\label{ssec: extremals}
In \cite{carlenloss} for $p=2$ and \cite{MN} for $p\in (1,n)$, extremal functions of \eqref{eqn: trace Sobolev} (equivalently, minimizers of \eqref{e: phi T}) were characterized as follows. For each $T>0$, the family $\mathcal{M}_T$ of extremals is given by
\[
\mathcal{M}_{T} = \left\{\pm \alpha^{-n/p^*} U_T( (\cdot - x_0)/\alpha) : x_0\in \partial H, \alpha>0 \right\}
\]
where the function $U_T$ is defined as follows.
For a given function $f:\R^n \to \R$ and $s \in \R$, let $\tau_sf(x) = f(x-se_1)$. Recall from the introduction that $T_E>0$ is the unique $T$ for which $\Phi_H(T_E)=E_{n,p}T_E$, and thus is the unique $T$ for which the linear lower bound $\Phi_H(T) \geq E_{n,p}T$  implied by the Escobar inequality
\begin{equation}
\label{eqn: escobar}
 \| \nabla u\|_{L^p(H)} \geq E_{n,p} \| u\|_{L^{p^\sharp}(\partial H)} \qquad \text{ for all } u \in \dot{W}^{1,p}(H)
\end{equation}
is saturated; see Figure~\ref{fig1}. 

\begin{itemize}
    \item For $T \in\left(0, T_E\right)$,  there is a unique $s_T \in \mathbb{R}$ such that
\begin{equation}\label{eqn: Sobolev minimizers}
U_T = \frac{\tau_{s_T} U_S}{\left\|\tau_{s_T} U_S\right\|_{L^{p^{\star}}(H)}}\chi_{\overline{H}} \qquad \text{ where } \qquad U_S(x) = \Big( 1 + |x|^{{p}/{(p-1)}}\Big)^{{(p-n)}/{p}}\,.
\end{equation}
The function $U_S$ is the unique (modulo symmetries) extremal for the Sobolev inequality on $\R^n$. When $p=2$, up to a constant multiple, the conformal metric $U_T^{4/(n-2)} g_{\rm euc}$ on $H$ is isometric to a geodesic ball on the round sphere, with radius tending to zero as $T \to T_E$ and to the diameter of the sphere as $T\to 0.$ 

\item  For $T=T_E$, i.e. the point corresponding to the Escobar trace inequality, 
\begin{equation}
    \label{eqn: escobar functions}
U_T = \frac{\tau_{s_T} U_E}{\left\|\tau_{s_T} U_E\right\|_{L^{p^{\star}}(H)} }\chi_{\overline{H}} \qquad \text{ where } \qquad U_E(x) = |x|^{(p-n)/(p-1)}\,
\end{equation}
with $s_T=-1$.  The function $U_{T_E}$ is the unique (modulo symmetries) extremal function for the Escobar trace inequality. Replacing $s_T=-1$ by any other $s<0$ in \eqref{eqn: escobar functions} gives a dilation of $U_T$ and hence another extremal.
When $p=2$, the conformal metric $U_{T_E}^{4/(n-2)} g_{\rm euc}$ on $H$ is isometric to a ball in $\R^n$.

\item For $T>T_E$, there is a unique $s_T<-1$ such that
\begin{equation}
    \label{eqn: BE family}
 U_T = \frac{\tau_{s_T } U_{B E}}{\left\|\tau_{s_T } U_{B E}\right\|_{L^{p^{\star}}(H)}}\chi_{\overline{H}} , \qquad \text { where }\qquad U_{B E}(x)=\left(|x|^{p /(p-1)}-1\right)^{(p-n) / p}.
\end{equation}
When $p=2$, up to a constant multiple, the conformal metric $U_{T}^{4/(n-2)} g_{\rm euc}$ on $H$ is isometric to a geodesic ball in hyperbolic space, with radius tending to zero as $T\to T_E$ and to infinity as $T\to \infty$.
\end{itemize}
No minimizers in \eqref{e: phi T} exist for $T=0$, since the support of the extremals for the Sobolev inequality is all of $\R^n$.

In \cite{MN} we establish various properties of the function $T\mapsto \Phi_H(T)$. First, let $T_0 \in (0,T_E)$ be the 
$L^{p^\sharp}(\partial H)$ norm of \eqref{eqn: Sobolev minimizers} with $s_{T_0}=0$. Then $\Phi_H$ uniquely achieves its global minimum $\Phi_H\left(T_0\right)=S_{n, p} / 2^{1 /n}$  at $T_0$, saturating the constant lower bound $\Phi_H(T) \geq S_{n,p}/2^{1/n}$ (see Figure~\ref{fig1}) resulting from the inequality
 \begin{equation}\label{eqn: gradient domain}
      \| \nabla u\|_{L^p(H)} \geq \frac{S_{n,p}}{2^{1/n}} \| u\|_{L^{p^*}(H)} \qquad \text{ for all } u \in \dot{W}^{1,p}(H)\,.
 \end{equation}
The inequality \eqref{eqn: gradient domain} is a direct consequence of the Sobolev inequality on $\R^n$ and reflection.

We additionally show in \cite{MN} that $T\mapsto \Phi_H(T)$ is strictly decreasing on $(0,T_E)$, concave on $(0,T_*)$ for some $T_*\in (0,T_0)$, and strictly increasing and convex on $(T_0, \infty)$. A simple divergence theorem computation shows that 
\begin{equation}
    \label{eqn: asymptotic}
    \Phi_H(T)>\frac{T^{p^{\sharp}}}{p^{\sharp}}\qquad \text{ for every } T>0,
\end{equation} 
and direct analysis of \eqref{eqn: BE family} shows this lower bound is asymptotically saturated as $T \rightarrow \infty$ (see Figure~\ref{fig1}).

In \cite{maggivillaniJGA}, Maggi \& Villani show that for any open, connected Lipschitz domain $\Omega\subset \R^n$, $\Phi_\Omega(T) \geq \Phi_B(T)$ for $T \in (0,{\rm ISO}(B_1)^{1/p^\sharp})$, where ${\rm ISO}(B_1) = n |B_1|$ and $\Phi_\Omega(T)$ is the analogous minimization problem to \eqref{e: phi T} with $\Omega$ in place of $H$. The lower bound $\Phi_H(T)\geq \Phi_B(T)$ on this interval is not saturated (see Figure~\ref{fig1}); the question of whether other domains can have $\Phi_\Omega(T)=\Phi_B(T)$ was investigated in \cite{MNT}. Note that any open Lipschitz domain $\Omega$ enjoys the complementary upper bound $\Phi_\Omega(T) \leq \Phi_H(T).$

\subsection{Mass transportation background}


We briefly recall some facts from optimal transport theory, and refer the reader to \cite{VillaniBook, MaggiOTBook} for further introduction.
For a $\mu$ (Borel) probability measure on $\R^n$ and a Borel measurable map $\TM: \mathbb{R}^n \rightarrow \mathbb{R}^n$, the pushforward of $\mu$ through $T$ is the probability measure $\TM_\# \mu$ defined by
\begin{equation}
    \label{eqn: transport cond sets}
\TM_\# \mu(A)=\mu\left(\TM^{-1}(A)\right) \quad \text{ for all } A \subset \mathbb{R}^n .
\end{equation}
By approximation, this means that 
\begin{equation}  \label{eqn: transport cond}
 \int_{\mathbb{R}^n} \xi \, d \TM_\# \mu=\int_{\mathbb{R}^n} \xi \circ \TM \,d \mu
\end{equation}
for every Borel measurable function $\xi: \mathbb{R}^n \rightarrow[0, \infty]$.

Now suppose $\mu =F \mathcal{L}^n$ and $\nu =G  \mathcal{L}^n$ are absolutely continuous probability measures on $\mathbb{R}^n$. By the Brenier-McCann theorem \cite{Brenier91, McCann97} (see \cite[Cor. 2.30]{VillaniBook}), there is a 
convex function $\varphi: \mathbb{R}^n \rightarrow \mathbb{R} \cup\{+\infty\}$ such that the map $\TM = \nabla \varphi$ is defined $\mu$-a.e. and satisfies
$$\TM_\#\mu =\nu.$$
The map $\TM$ is uniquely determined $\mu$-a.e., and is called the Brenier map from $\mu$ to $\nu$. It is the unique optimal transport map for the quadratic cost, though we will not use this. 
If $\varphi$ is $C^2$, then \eqref{eqn: transport cond} and the area formula give 
\begin{equation}
    \label{eqn: MA}
F(x)=G(\nabla \varphi(x)) \operatorname{det} \nabla^2 \varphi(x) \qquad \mu\text{-a.e.};
\end{equation}
the same identity holds in general for  the Alexandrov Hessian $\nabla^2\varphi$, i.e. for the absolutely continuous part of the distributional Hessian \cite{McCann97}.
 \smallskip

A key property of the Brenier map is the cyclical monotonicity of its graph. A subset $\Gamma \subset \mathbb{R}^n \times \mathbb{R}^n$ is {\it cyclically monotone} if  
\[
\sum_{i=1}^m y_i \cdot (x_{i+1}- x_i) \leq 0
\]
for every $m \in \mathbb{N}$ and every collection of $m$ points $(x_1, y_1),\dots, (x_m,y_m)$ in $\Gamma$,
with the convention that $x_{m+1}=x_1$. 
There is a set $A$ of full $\mu$ measure (i.e. $\mu(A)=1$) such that the graph $\Gamma = \{(x,\TM(x)) : x \in A\}$ of the Brenier map is cyclically monotone, see, e.g., \cite[Prop. 2.24]{VillaniBook}\footnote{More generally, the support of the optimal plan between any two probability measures on $\R^n$ is cyclically monotone, and this fact is used in one proof of the Brenier-McCann theorem.}
Applying this fact for $m=2$ guarantees that 
\begin{equation}
    \label{eqn: cyclical monotonicity}
(\TM(x_1)-\TM(x_2))\cdot (x_1 -x_2)  \geq 0\, \quad \text{ for all }(x_1, x_2)\in A\times A .
\end{equation}

\subsection{Mass transportation argument}\label{ssec:OT proof}
Let us sketch the mass transportation proof of \eqref{eqn: trace Sobolev} given in \cite{MN}. It is a variant of the mass transportation arguments used to prove the Sobolev inequality \cite{CENV2004} and various other functional inequalities, and is especially inspired by Nazaret's proof of the Escobar inequality \cite{Nazaret2006}.
Fix $T>0$ and let $s=s_T \in \R$ be as in  section~\ref{ssec: extremals}. 
Direct computation verifies the identity
\begin{equation}
    \label{eqn: LHS OT arg}
p^{\sharp} \|\nabla U_T\|_{L^p(H)} Y_T+s T^{p^{\sharp}} =n \int_H U_{T}^{p^{\sharp}} d x\, \qquad\text{ where }\qquad Y_T= \Big(\int_H U_T^{p^{\star}}\left|x-s e_1\right|^{p/(p-1)} d x\Big)^{(p-1)/p} \,.
\end{equation}

Now, fix $u \in  \compset \cap C^1_c(\overline{H})$ with $u \geq 0$. We aim to show that
\begin{equation}\label{eqn: goal}
    \| \nabla u\|_{L^p(H)}\geq \|\nabla U_T\|_{L^p(H)}\,.
\end{equation}
Letting $F = u^{p^*}$ and $G= U_T^{p^*}$, consider the measures
\[
\mu = u^{p^*}\, \mathcal{L}^n= F\, \mathcal{L}^n \qquad \text{ and }\qquad \nu = U_T^{p^*} \, \mathcal{L}^n = G\, \mathcal{L}^n\,,
\]
and let $\TM = \nabla \varphi$ be the Brenier map from $\mu$ to $\nu.$ Applying the transport condition \eqref{eqn: transport cond}, the identity \eqref{eqn: MA}, and the arithmetic-geometric mean inequality to the nonnegative eigenvalues of the (Alexandrov) Hessian $D^2\varphi$,\footnote{As with \eqref{eqn: MA}, there is some subtlety when $\varphi$ is not $C^2$; in \eqref{eqn: OT arg 1}, AM--GM is applied $\mu$-a.e. to the eigenvalues of the Alexandrov Hessian, and  by convexity, the Alexandrov Laplacian of $\varphi$ is bounded above by its distributional Laplacian.} we find
\begin{equation}\label{eqn: OT arg 1}
\int_H U_T^{p^{\sharp}}=\int_{\mathbb{R}^n} G^{1-1 / n}=\int_{\mathbb{R}^n} G(\nabla \varphi)^{-1 / n} F=\int_{\mathbb{R}^n}\left(\operatorname{det} \nabla^2 \varphi\right)^{1 / n} F^{1-1 / n} \leq \frac{1}{n} \int_{\mathbb{R}^n} F^{1-1 / n} d(\operatorname{div} \TM)\,.
\end{equation}
On the right-hand side, we subtract the divergence-free vector field $se_1$ from $\TM$, letting $\SM=\TM-se_1$. Then, applying the divergence theorem, we obtain
\begin{equation}
    \label{eqn: OT arg 2}
\int_{\mathbb{R}^n} F^{1-1 / n} d(\operatorname{div} \TM)=\int_H u^{p^{\sharp}} d(\operatorname{div} \SM)=-p^{\sharp} \int_H u^{p^{\sharp}-1} \nabla u \cdot \SM d x-\int_{\partial H} u^{p^{\sharp}} \SM \cdot {e}_1 d \mathcal{H}^{n-1} .
\end{equation}
Since $\TM$ transports $\mu$ to $\nu$ and $\operatorname{spt}(\nu) \subset \overline{H}$, we have $\SM(x) \cdot\left(-{e}_1\right) \leq s$ for $\mathcal{H}^{n-1}$-a.e. $x \in \operatorname{spt}(\mu) \cap \partial H$. Thus, in summary, \eqref{eqn: OT arg 1} and \eqref{eqn: OT arg 2} yield
\begin{equation}\label{eqn: OT arg intermed}
n \int_H U_T^{p^{\sharp}}\,dx \leq-p^{\sharp} \int_H u^{p^{\sharp}-1} \nabla u \cdot\left(\TM-s e_1\right)\,dx +s\, T^{p^\sharp}\,.
\end{equation}
Now we bound the first term on the right-hand side. Using Cauchy-Schwarz and H\"{o}lder's inequalities and the transport condition \eqref{eqn: transport cond}, we find
\begin{align}
\label{eqn:CS} -p^{\sharp} \int_H u^{p^{\sharp}-1} \nabla u \cdot\left(\TM-s e_1\right)\,dx& \leq p^{\sharp} \int_H u^{p^{\sharp}-1} |\nabla u| \cdot |\TM-s e_1|\,dx\\
\nonumber& \leq p^{\sharp}\|\nabla u\|_{L^p(H)}\left(\int_H u^{p^{\star}}\left|\TM(x)-s e_1\right|^{p/(p-1)} d x\right)^{(p-1)/p} \\
\nonumber& =p^{\sharp}\|\nabla u\|_{L^p(H)}\left(\int_H U_T^{p^{\star}}\left|x-s e_1\right|^{p/(p-1)} d x\right)^{(p-1)/p} =p^{\sharp}\|\nabla u\|_{L^p(H)}\,Y_T.
\end{align}
Combining this with \eqref{eqn: LHS OT arg} and \eqref{eqn: OT arg intermed} shows that 
\begin{equation}
    \label{eqn: final}
p^{\sharp} \|\nabla U_T\|_{L^p(H)} Y_T+s T^{p^{\sharp}}
\leq p^{\sharp}\|\nabla u\|_{L^p(H)}\,Y_T + s T^{p^{\sharp}}
\end{equation}
which directly implies \eqref{eqn: goal}.

The inequality \eqref{eqn: final} sandwiches each individual inequality in the proof, including \eqref{eqn:CS}, so
\begin{equation*}
    \begin{split}
(\|\nabla u\|_{L^p(H)} - \|\nabla U_T\|_{L^p(H)})\, Y_T \geq 
&\int_H u^{p^{\sharp}-1}\Big( |\nabla u| \cdot |\TM-s e_1|
- (-\nabla u) \cdot\left(\TM-s e_1\right)\Big)\,dx\\
=\frac{1}{2}& \int_H u^{p^{\sharp}-1}|\nabla u|\left|\TM-s e_1\right|\left|\frac{-\nabla u}{|\nabla u|}-\frac{\left(\TM-s e_1\right)}{\left|\TM-s e_1\right|}\right|^2\,dx.
 \end{split}
\end{equation*}
Since $\|\nabla u\|_{L^p(H)} - \|\nabla U_T\|_{L^p(H)} \leq C\delta_T(u)$ with $C= 1/(p\Phi_H(T)^{p-1}),$
this means there is a constant $C_{n,p,T}>0$ such that 
\begin{equation}    \label{eqn: deficit and CS}
    C_{n,p,T}\delta_T(u) \geq \int_H u^{p^{\sharp}-1}|\nabla u|\left|\TM-s e_1\right|\left|\frac{-\nabla u}{|\nabla u|}-\frac{\left(\TM-s e_1\right)}{\left|\TM-s e_1\right|}\right|^2\,dx.
\end{equation}
In the proof of Theorem~\ref{thm: strict energy gap}, we will only use \eqref{eqn: deficit and CS} and the fact that $U_T$ is radially symmetric and decreasing about a point $se_1$.

\section{Proof of Theorem~\ref{thm: strict energy gap}}
In this section we prove Theorem~\ref{thm: strict energy gap}. As described in the introduction, the idea is to construct a function $w$, which is essentially the sum of a copy of $\vv_1 U_{T_1}$ centered at $2^{R+2}e_n$ and a copy of $\vv_2U_{T_2}$ centered at $-2^{R+2}e_n$ for $R\gg 1$, such that $\int|\nabla w|^p \leq  \vv_1^p \Phi_H\left(T_1\right)^p+\vv_2^p \Phi_H\left(T_2\right)^p + \varepsilon$. We are able to estimate the corresponding Brenier map $\TM$ in a somewhat explicit manner. In particular, the cyclical monotonicity of the graph of $\TM$ in the form \eqref{eqn: cyclical monotonicity} shows that  $\TM$ maps most points in the support of the translated copy of $\vv_1 U_{T_1}$ to $\{y\cdot e_n>0\}$, forcing a definite lower bound for the right-hand side of the estimate 
\eqref{eqn: deficit and CS} and thus showing $\delta_T(w) \geq 2\varepsilon$. 
\begin{proof}[Proof of Theorem~\ref{thm: strict energy gap}]
{\it Step 1:} We begin by fixing parameters and notation. 
Without loss of generality assume $\vv_1 \leq \vv_2$. As above, let $T_1={\ttt_1}/{\vv_1}\geq 0$ and $T_2={\ttt_2}/{\vv_2}\geq 0$. 
We will use the convention that if $T_i=0$, then $U_{T_i}=U_0$, where we let $U_0 = U_S/\|U_S\|_{L^{p^*}(\R^n)}$ for the standard Talenti bubble $U_S$ defined in \eqref{eqn: Sobolev minimizers}.
From the explicit form of $U_{T_1}$ (recall \eqref{eqn: Sobolev minimizers}-\eqref{eqn: BE family}), it is not difficult to see that there exists $\bar{c}=\bar{c}(n, p,T_1)>0$ small enough such that the set
\begin{equation}\label{eqn: G def}
\G=\left\{x \in H\ : \ U_{T_1}(x) \geq \bar{c}\,\left\|U_{T_1}\right\|_{L^\infty(H)} , \quad \left|\nabla U_{T_1}(x)\right| \geq \bar{c}, \quad\mfrac{\nabla U_{T_1}(x) \cdot e_n}{|\nabla U_{T_1}(x)|} \geq \bar{c}\, \right\}
\end{equation}
is nonempty.  Note that $\G \subset B_\rho$ for some $\rho=\rho(n,T_1)>0$. We define $\bar{a}=\bar{a}(n,T_1, \vv_1)$ by
\begin{equation}
    \label{eqn: a def}
3 \bar{a}:=\vv_1^{p^*} \int_\G U_{T_1}^{p^*}>0 .
\end{equation}
Let $R= R(n, T, \vv_1, \vv_2, \ttt_1,\ttt_2)\geq 2\rho >0$ be a large fixed number to be specified later in the proof. Let
\begin{equation}\label{eqn: flat cone}
\tK=\left\{z \in \mathbb{R}^n:\left|z_n\right|<\mfrac{R}{2^R}\left|z^{\prime}\right|\right\}  \,. 
\end{equation}
Here $z'$ denotes the projection of $z$ onto $\R^{n-1}\times\{0\} \subset \R^n$. Since $|\tK \cap B_r| = o_R(1)$ for any fixed $r>0$, we can take $R$ large enough so that 
\[
\bar{b}:=\int_{\tK\cap H}U_T^{p^*}\,dx \leq \bar{a}\,.
\]
Here and below, $o_R(1)$ is a number whose absolute value can be made arbitrarily small by taking $R$ sufficiently large.
\\

{\it Step 2:} Next we construct the main function $w =w_R \in \compset$. For simplicity we first assume $T_i>0$ for both $i=1,2$. Let $\eta: \R^n \to \R$ be a smooth nonnegative cutoff function  supported in $B_1$ with $0\leq \eta\leq 1$ on $\R^n$ and $\eta=1$ in $B_{1/2}$. The function ${W}_1(x)=\vv_1 U_{T_1}(x) \eta(\frac{x}{R})$ satisfies
\begin{align*}
\int_H {W}_1^{p^*}\,dx & =\vv_1^{p^*}+o_R(1), \qquad \int_{\partial H} {W}_1^{p^{\sharp}}\,d\mathcal{H}^{n-1} =\vv_1^{p^\sharp} T_1^{p^{\sharp}} +o_R(1), \qquad
\int_H\left|\nabla {W}_1\right|^p \,dx =\vv_1^p \Phi_H\left(T_1\right)^p+o_R(1)\,.
\end{align*}
The analogous estimates hold for ${W}_2(x)=\vv_2 U_{T_2}(x) \eta(\frac{x}{R})$. 
By construction, we have $\int_H {W}_i^{p^*}\,dx \leq \vv_i^{p^*}$ and $ \int_{\partial H} {W}_i^{p^\sharp} \, d\mathcal{H}^{n-1}\leq (\vv_iT_i)^{p^\sharp}=\ttt_i^{p^\sharp}$ for $i=1,2$. So, we may choose a nonnegative smooth function $\psi_R: H \to \R$ supported in $B_R$ so that 
\begin{equation}
    \label{eqn: w def}
w(x)={W}_1\left(x-2^{R+2} e_n\right)+\left({W}_2\left(x+2^{R+2} e_n\right)+\psi_R\left(x+2^{R+2} e_n\right)\right)
\end{equation}
lies in $\compset$ and satisfies 
\begin{equation}\label{eqn: w energy}
\int_H|\nabla w|^p\,dx = \vv_1^p \Phi_H\left(T_1\right)^p+\vv_2^p \Phi_H\left(T_2\right)^p + o_R(1)\,.
\end{equation}
The support of $w$ is contained in $B_R^{+} \cup B_R^{-}$ where we let
\begin{equation}
    \label{eqn:BR}
B_R^{+}=B_R\left(2^{R+2} e_n\right) \cap \overline{H}, \qquad B_R^{-}=B_R\left(-2^{R+2} e_n\right)\cap \overline{H} .
\end{equation}
Recalling that $\psi_R\ge 0$ and $\vv_1\leq \vv_2$, we have
\begin{equation}\label{eqn: half}
  \int_{B_{R}^-} w^{p^*} \,dx \geq \frac{1}{2} .  
\end{equation}

In the case that $T_i=0$ for either $i=1$ or $2$, the construction is slightly modified. For this $i$ only, instead define the function $W_i$ by $W_i(x) = \vv_i U_{0}(x-Re_1) \eta(\frac{x-Re_1}{R})$. Then let $w$ be defined as in \eqref{eqn: w def}. Once more $w\in \compset$ and satisfies \eqref{eqn: w energy}. If $T_1=0$, define the set $B_R^{+}$ by $B_R(2^{R+2}e_n +Re_1)\cap\overline{H}$ and let $B_R^-$ be as 
in \eqref{eqn:BR}. If instead $T_2=0$, let $B_R^+$ be as in \eqref{eqn:BR} and set $B_R^-=(B_R(-2^{R+2}e_n) \cup B_R(-2^{R+2}e_n +Re_1))\cap\overline{H}$. Then once more in this case, $w$ is supported in $B_R^+\cup B_R^-$ and \eqref{eqn: half} holds.

{\it Step 3:} Let $\mu=w^{p^*}\mathcal{L}^n$ and $\nu=U_T^{p^*} \mathcal{L}^n$ and let $\TM$ be the Brenier map from $\mu$ to $\nu$.
Consider the sets 
\begin{align}
E&=\left\{x \in B_R^{+}: \TM(x) \cdot e_n<0\right\}=\TM^{-1}\left(\left\{y_n<0\right\}\right) \cap B_R^{+},\\
\label{eqn: F} F&= \left\{x \in B_R^- : \TM(x) \cdot e_n \geq 0\right\} =  \TM^{-1}\left(\{y_n \geq 0\}\right) \cap B_R^{-}
\end{align}
of points in $B_R^\pm$ that get mapped across the plane $\{y_n=0\}$ by $\TM$.
We claim that
\begin{equation}\label{eqn: mu E}
    \mu(E) \leq \bar{b}.
\end{equation}
As above $\bar{b}=\int_{\tK\cap H}U_T^{p^*}\,dx = \nu(\tK)$.  Suppose not, i.e. $\mu(E)>\bar{b}$. From the symmetry of $U_T$ and the transport condition \eqref{eqn: transport cond sets}, we have 
\begin{align*}
\frac{1}{2}=\nu\left(\left\{y_n<0\right\}\right) & =\mu\left(\TM^{-1}\left(\left\{y_n<0\right\}\right)\right) \\
& =\mu\left(\TM^{-1}(\left\{y_n<0\right\}) \cap B_R^{-}\right)+\mu(E) > \mu\left(\TM^{-1}\left(\left\{y_n<0\right\}\right) \cap B_R^{-}\right)+ \bar{b}\,.
\end{align*}
That is,
 $\mu(\TM^{-1}(\{y_n<0\}) \cap B_R^{-}) < \frac{1}{2}-\bar{b}$. 
Since $\mu\left(B_R^{-}\right) \geq\frac{1}{2}$ by \eqref{eqn: half}, this means the set $F$ defined in \eqref{eqn: F} has $
\mu(F) > \bar{b}$ as well.
Now,  with $\tK$ as in \eqref{eqn: flat cone}, let
\begin{align*}
F_*& =\{ x \in B_R^{-}:  \TM(x) \in \{y\cdot e_n>0\} \setminus \tK\} \subset F\,,\\
E_* &=\{x \in B_R^{+}: \TM(x) \in\{y\cdot e_n<0\} \setminus \tK\} \subset E\,
\end{align*}
be the sets of points in $B_R^\pm$ mapping even further to the ``wrong side'' of the plane $\{y_n=0\}$. Since $F \setminus F_*\subset \TM^{-1}(\{y_n \geq 0\} \cap \tK)$, we have
$$
\mu(F \setminus F_*) \leq \mu\big(\TM^{-1}(\{y_n \geq 0\} \cap \tK)\big)=\nu(\{y_n \geq 0\} \cap \tK)=\frac{\bar{b}}{2}\,.
$$
The final identity comes from the reflection symmetry of $\tK$ and $U_T$ across $\{x_n=0\}$. Since $\mu(F)=\mu\left(F_*\right)+\mu\left(F \setminus F_*\right)$, we find $
\mu\left(F_*\right) \geq \frac{\bar{b}}{2}$, and analogous reasoning shows $\mu\left(E_*\right)>\frac{\bar{b}}{2}$. In particular, letting $A$ be the full $\mu$-measure set on which \eqref{eqn: cyclical monotonicity} holds, the sets $E_*\cap A$ and $F_*\cap A$ are nonempty.

Now, take any $x_1 \in E_*\cap A \subset B^{+}_R$ and $x_2 \in F_*\cap A \subset B^{-}_R$. Then by construction $x_1-x_2$ lies in the positive cone
$$
\begin{aligned}
\K & =\Big\{z \in \mathbb{R}^n : z_n>\mfrac{2^R}{R}\left|z^{\prime}\right|\Big\};
\end{aligned}
$$
note that this holds even in the case when one of the $T_i$ is zero.
Observe that  $z \cdot y <0$ for any $z \in \K$ 
and $y \in\left\{y_n<0\right\} \setminus \tK$.
 So, by cyclical monotonicity \eqref{eqn: cyclical monotonicity},
\begin{equation}\label{eqn: CM consequence}
 \TM(x_1)-\TM(x_2) \notin\left\{y_n<0\right\} \setminus \tK\,.
\end{equation}
On the other hand, from the definition of $E_*$ we have $\TM(x_1) \in\{y_n<0\} \setminus \tK$. Similarly, from the definition of $F_*$, we have
$\TM(x_2)\in \{y_n>0\} \setminus \tK$ and so $-\TM(x_2) \in\left\{y_n<0\right\} \setminus \tK$. Since $\{y_n<0 \}\setminus \tK$ is a convex cone, 
\[\TM(x_1)-\TM(x_2) \in\{y_n<0\} \setminus \tK.
\]
This contradicts \eqref{eqn: CM consequence}. Thus \eqref{eqn: mu E} holds.\\

{\it Step 4:}
Since $\TM$ is a transport map and $U_T$ is bounded, for any $\varepsilon>0$,
\begin{equation}\label{eqn: excise ball}
\mu\left(\left\{x:\left|\TM(x)-se_1\right|<\varepsilon\right\}\right) 
=\nu(B(se_1, \varepsilon)) \leq C \varepsilon^n
\end{equation}
where $C= C(n,p,T)$.
Choose $\varepsilon>0$ small enough so $C \varepsilon^n \leq \bar{a}$ with $\bar{a}$ as in \eqref{eqn: a def}. (In the case $s<0$ we can choose $\varepsilon$ so that $\nu(B(se_1, \varepsilon))=0$.) Note that since $R\geq 2\rho$, if $T_1>0,$ then for $\G$ as in \eqref{eqn: G def},
$$\mu(\G +2^{R+2} e_n)=\int_{\G+2^{R+2} e_n}w^{p^*} dx
=\int_\G\left(\vv_1 U_{T_1}\right)^{p^*}=3 \bar{a}.$$
If $T_1=0,$ the same conclusion holds with $\G+2^{R+2} e_n + Re_1$ in place of $\G+2^{R+2} e_n$.
So, thanks to \eqref{eqn: excise ball}, \eqref{eqn: mu E}, and $\bar{b} \leq \bar{a}$, we have  $\mu(\G_*) \geq \bar{a}$. Here we let
$$
\G_*=\big( \G+ 2^{R+2} e_n\big) \setminus
\big(\{x:|\TM(x)-s e_1|<\varepsilon\} \cup\{x: \TM(x) \cdot e_n<0\}\big) \,
$$
when $T_1>0$, and when $T_1=0$, we define $\G_*$ identically except with $\G+ 2^{R+2} e_n$ replaced by $\G+ 2^{R+2} e_n + Re_1$.
From the definition of $\G$ in \eqref{eqn: G def}, we thus have $|\frac{-\nabla w}{|\nabla w|}-\frac{\TM-s e_1}{\left|\TM-s e_1\right|}|^2 \geq \bar{c}^2$ on $\G_*$. So, recalling \eqref{eqn: deficit and CS} and noting that $p^\sharp-1=p^*-\frac{n}{n-p}$,  we have
\begin{equation}
    \label{eqn: deficit big}
    \begin{split}
C_{n,p,T}  \delta_T(w) & \geq \int_{\G_*} w^{p^{\sharp}-1}|\nabla w| \cdot\left|\TM-s e_1\right|\left|\frac{-\nabla w}{|\nabla w|}-\frac{\TM-s e_1}{\left|\TM-s e_1\right|}\right|^2\,dx \\
&= \int_{\G_*} w^{-{n}/({n-p})}|\nabla w| \cdot\left|\TM-s e_1\right|\left|\frac{-\nabla w}{|\nabla w|}-\frac{\TM-s e_1}{\left|\TM-s e_1\right|}\right|^2\,d\mu \\
& \geq \left\|U_{T_1}\right\|_{L^{\infty}(H)}^{-n/(n-p)} \cdot \bar{c} \cdot \varepsilon \int_{\G_*}\left|\frac{-\nabla w}{|\nabla w|}-\frac{\TM-s e_1}{\left|\TM-s e_1\right|}\right|^2\,d\mu  \geq \bar{a} \bar{c}^{3}\left\|U_{T_1}\right\|_{L^\infty(H)}^{-n/(n-p)} \varepsilon=:2c_0 .      
    \end{split}
\end{equation}
Up to possibly further increasing $R$ depending on $n,p,T_1,\bar{a}, \bar{c}$ and $\varepsilon$, and thus on $n,p, T, \vv_1,\vv_2, \ttt_1,\ttt_2$, in \eqref{eqn: w energy} we may take the error $o_R(1)$ to be at most $c_0$, so absorbing it yields
\[
\vv_1^p \Phi_H\left(T_1\right)^p+\vv_2^p \Phi_H\left(T_2\right)^p  - \Phi_H(T)^p \geq c_0.
\]
This completes the proof.
\end{proof}

\section{Proofs of Theorem~\ref{thm: qual stability} and Corollary~\ref{cor: quant stability}}
The strict binding inequality Theorem~\ref{thm: strict energy gap} is the main tool toward proving Theorem~\ref{thm: qual stability}. To complete the proof of Theorem~\ref{thm: qual stability}, let us recall the first and second concentration-compactness lemmas in the present setting.
\begin{lemma}[Concentration-Compactness Lemma I]
    \label{lem: CC1}
    Let $\{\nu_k\}_k$ be a sequence of probability measures on $\mathbb{R}^n$. There is a subsequence (unrelabeled) such that one of the following three conditions holds:
    \begin{enumerate}
        \item (Compactness) There is a sequence $\{x_k\} \subset \R^n$ such that for any $\varepsilon>0$, there is a radius $R>0$ such that $\nu_k (B_R(x_k)) \geq 1-\varepsilon $ for all $k$.
        \item (Vanishing) For all $R>0$, $\lim_{k\to \infty} (\sup_{x \in \R^n} \nu_k(B_R(x)))=0$.
        \item (Dichotomy) There is a number $\lambda \in (0,1)$ such that for all $\varepsilon>0$, there exist $R>0$ and $\{x_k\}\subset \R^n$ such that for all $R'>R,$
        \begin{align*}
     \limsup_{k\to \infty} \Big( |\lambda - \nu_k^1(\R^n)| + |(1-\lambda) -\nu_k^2(\R^n)|\Big) \leq \varepsilon
            \end{align*}
            for the measures 
  \[
  \nu_k^1 = \nu_k \llcorner B_{R'}(x_k) \qquad \text{ and } \qquad \nu_k^2 = \nu_k \llcorner \R^n \setminus B_{8R'}(x_k)\,.
\]
    \end{enumerate}
\end{lemma}
Compared to the usual statement (see \cite[Lemma I.1]{CCLocallyCompact} or \cite[Lemma 4.3]{Struwe}) the statement of Lemma~\ref{lem: CC1},
 we spell out the splitting measures $\nu_k^{1},\nu_k^{2}$  in the dichotomy alternative more explicitly. The statement given above is already present in the proof of the standard statement.

\begin{lemma}[Concentration-Compactness Lemma II]\label{lem: CC2} Let $n \geq 2$ and $ p \in(1, n)$. 
If $\left\{u_k\right\}_k$ is a sequence in $L_{\text {loc }}^1(H),\left\{\nabla u_k\right\}_k$ is bounded in $L^p\left(H ; \mathbb{R}^n\right)$ and $u_k \rightharpoonup u$ as distributions in $H$, then the Radon measures on $\overline{H}$ defined by
$$
\mu_k=\left|\nabla u_k\right|^p \mathcal{L}^n\llcorner H, \quad \nu_k=\left|u_k\right|^{p^{\star}} \mathcal{L}^n\llcorner H, \quad \tau_k=\left|u_k\right|^{p^{\sharp}} \mathcal{H}^{n-1}\llcorner\partial H
$$
have subsequential weak-star limits $\mu, \nu$ and $\tau$ which satisfy

$$
\begin{aligned}
\nu & =|u|^{p^{\star}} \mathcal{L}^n\llcorner H+\sum_{i \in I} \vv_i^{p^{\star}} \delta_{x_i}, \\
\tau & =|u|^{p^{\sharp}} \mathcal{H}^{n-1}\llcorner\partial H+\sum_{i \in I} \ttt_i^{p^{\sharp}} \delta_{x_i}, \\
\mu & \geq|\nabla u|^p \mathcal{L}^n\llcorner H+\sum_{i \in I} \mathrm{~g}_i^p \delta_{x_i},
\end{aligned}
$$
where $\left\{x_i\right\}_{i \in I} \subset \bar{H}$ is an at most countable set, $\vv_i>0$ and $\ttt_i \geq 0$ for every $i \in I$, with $\ttt_i>0$ only if $x_i \in \partial H$, and

$$
\mathrm{g}_i \geq \vv_i \Phi_H\left(\frac{\ttt_i}{\vv_i}\right), \quad \forall i \in I\,.
$$
In particular, $\mathrm{g}_i \geq S_{n,p} \vv_i$ whenever $x_i \in H$.
\end{lemma}

This form of the second concentration-compactness lemma, accounting for the trace term, was shown in \cite[Lemma 2.1]{MNT}. Its proof is a basic adaptation of the classical version on $\R^n$, see \cite[Lemma I.1]{lions1985} or \cite[Lemma 4.8]{Struwe}.
 In \cite{MNT}, the lemma is stated on an open bounded domain $\Omega$ with $C^1$ boundary rather than on $H$, but the boundedness of the domain is not used in the proof; the only modification is to replace inequalities (A.2) and (A.3) there by the inequalities \eqref{eqn: gradient domain} and \eqref{eqn: escobar} respectively.

Finally, in the proof of Theorem~\ref{thm: qual stability}, we will also use the following simple lemma.
\begin{lemma}\label{lem: small trace}
    Fix $n\geq 2$ and $p \in (1,n)$. There is a constant $C=C(n,p)$ such that the following holds.
Let $\varepsilon>0$, $R>0$ and $x_0 \in \R^n$, and suppose $u \in \dot{W}^{1,p}(H)$ has $\|u\|_{L^{p^*}(H)}=1$ and $ \|u\|_{L^{p^*}(H\cap B_{8R}(x_0)\setminus B_R(x_0))} \leq \varepsilon.$   Then
\[
\int_{\partial H \cap (B_{7R}(x_0) \setminus B_{2R}(x_0))} \! \! \!|u|^{p^\sharp} \, d\mathcal{H}^{n-1} \leq C(\varepsilon+\| \nabla u\|_{L^p(H)}) \varepsilon^{p^\sharp -1}.
\]
In particular, if $\|\nabla u\|_{L^p(H)}\leq \tilde{C}$, the right-hand side tends to zero as $\varepsilon \to 0$.
\end{lemma}
\begin{proof}
Choose a cutoff function $\psi:\R^n \to \R$  with $0 \leq \psi \leq 1$,  $\psi=1$ on $B_{7}\setminus B_{2}$, $\psi = 0$ on $B_{1} \cup (\R^n\setminus B_8)$, and let $\psi_R (x)  = \psi((x-x_0)/R)$. 
Then the function $v  = \psi_R u$ has $\|v\|_{L^{p^*}(H)}\leq \varepsilon$ and 
\[
\|\nabla v \|_{L^p(H)} \leq \|\nabla u\|_{L^p(H)} + \|u\nabla \psi_{R}\|_{L^p(H)}  \leq \|\nabla u\|_{L^p(H)} + \|u\|_{L^{p^*}(H\cap B_{8R}(x_0)\setminus B_R(x_0))} \| \nabla \psi_R\|_{L^n(\R^n)}\,.
\]
By scaling, $\| \nabla \psi_R\|_{L^n(\R^n)}=\| \nabla \psi\|_{L^n(\R^n)}$, so $\|\nabla v \|_{L^p(H)} \leq  \|\nabla u\|_{L^p(H)}+ C'\varepsilon$
for a constant $C'=C'(n)$. If $\|v\|_{L^{p^*}(H)}=0$, there is nothing to show. Otherwise, applying \eqref{eqn: asymptotic} to $v$ shows that 
\[
\|\nabla u\|_{L^p(H)}+ C'\varepsilon \geq \| v\|_{L^{p^*}(H)}\Phi_H\bigg(\mfrac{\|v \|_{L^{p^\sharp}(\partial H)}}{\| v\|_{L^{p^*}(H)}}\bigg) 
\geq \frac{\|v \|_{L^{p^\sharp}(\partial H)}^{p^\sharp}}{p^\sharp\| v\|_{L^{p^*}(H)}^{p^\sharp-1}} \geq \frac{\|v \|_{L^{p^\sharp}(\partial H)}^{p^\sharp}}{p^\sharp\varepsilon^{p^\sharp-1}},
\]
which, after rearranging, completes the proof.
\end{proof}
With these lemmas in hand, we can now prove Theorem~\ref{thm: qual stability}.
\begin{proof}[Proof of Theorem~\ref{thm: qual stability}]
    Let $\{\hat{u}_k\} \subset \compset$ be a sequence with $\delta_T(\hat{u}_k)\to0$. We begin by normalizing. For each $k \in \mathbb{N},$ choose $y_k \in \R^n$ and $R_k>0$ so 
    \begin{equation}
        \label{eqn: normalize 1}
       \sup_{y \in \R^n} \int_{H\cap B_{R_k}(y)} |\hat{u}_k|^{p^*}\, dx = \int_{H\cap B_{R_k}(y_k)} |\hat{u}_k|^{p^*}\,dx = \frac{1}{2}\,. 
      \end{equation}
By set containment, we may choose $y_k \in \overline{H}$.
Writing $y_k=(y_k^1, y_k')$, we set $\tilde{u}_k(\cdot) = \hat{u}_k(\cdot + (0, y_k'))$, so that $\tilde{u}_k \in\compset$ has $\delta_T(\tilde{u}_k)\to 0$ and satisfies \eqref{eqn: normalize 1} with $y_k^1 e_1$ in place of $y_k$. Next, consider the rescaled function $u_k \in \compset$ defined by
    $u_k(x) = R_k^{({n-p})/{p}} \tilde{u}_k({R_k}x)$, which has 
    \begin{equation}
        \label{eqn: normalize}
    \sup_{y \in \R^n} \int_{H\cap B_1(y)} |u_k|^{p^*}\, dx = \int_{H\cap B_1(d_ke_1)} |u_k|^{p^*}\,dx = \frac{1}{2}\,,
       \end{equation}
where $d_k = y^1_k/R_k \in [0,\infty).$ Once again, $\delta_T(u_k)=\delta_T(\hat{u}_k)\to0$ by scaling.

Apply Lemma~\ref{lem: CC1} to the sequence of probability measures $\nu_k = |u_k|^{p^*}\,\mathcal{L}^n\llcorner H$.
We will rule out both the vanishing alternative and the dichotomy alternative using Theorem~\ref{thm: strict energy gap} as follows.

{\it Vanishing does not occur.} We claim that $A:=\limsup_{k\to \infty} d_k <\infty$, which rules out the vanishing alternative, since in this case $\nu_k(B_{A+1}(0)) \geq 1/2$ for all $k$ large. Suppose by way of contradiction that, up to an un-relabeled subsequence, $d_k \to \infty.$ We will see this forces splitting of the sequence which is energetically expensive.

Let $J_k = \lfloor \log_2(d_k)\rfloor \in \mathbb{N}$, so that  $2^{J_k} \leq d_k \leq 2^{J_k+1}$. For at least one $1\leq j \leq J_k$, we have 
\begin{equation}
    \label{eqn: small mass}
\int_{H \cap B_{2^{j}}(d_ke_1) \setminus B_{2^{j-1}}(d_ke_1) }|u_k|^{p^*} \,dx \leq \frac{2}{\log_2(d_k)},
\end{equation}
as otherwise $\int_H |u_k|^{p^*}\,dx \ge 2J_k/\log_2(d_k)\ge 2J_k/(J_k+1)>1$ for $k$ large enough.

Let $\rho_k= 2^{j_k^*}$ where $j_k^*$ is the first $j$ for which \eqref{eqn: small mass} holds. Fix a smooth nonnegative cutoff function $\psi:\R^n\to [0,1]$ with supported in $B_1$ with $\psi = 1 $ in $B_{1/2}$. Let $\psi_k(\cdot ) = \psi((\cdot-d_ke_1)/ \rho_k)$ and $m_k = \|u_k \psi_k\|_{L^{p^*}(H)}$ and $\ell_k = \|u_k (1-\psi_k)\|_{L^{p^*}(H)}.$ Note that $m_k^{p^*} +\ell_k^{p^*} = 1+o_k(1)$ by \eqref{eqn: small mass} and our choice of $\rho_k$. Up to a subsequence (not relabeled), $m_k \to m_*$ and $\ell_k \to \ell_*$, where $ m_*^{p^*} + \ell_*^{p^*}=1$ and $m_*^{p^*} \in [1/2,1]$ by \eqref{eqn: normalize}. 

First suppose $m_*<1$ and thus $\ell_*>0.$ Using  $\psi_k^p+(1-\psi_k)^p \leq 1$, \eqref{eqn: small mass}, and H\"{o}lder's inequality just as in the proof of Lemma~\ref{lem: small trace}, we find that
\begin{align*}
\int_H |\nabla u_k|^p\, dx &\geq \int_H |\nabla(u_k \psi_{k})|^p \,dx + \int_H |\nabla (u_k(1-\psi_k))|^p \,dx +o_k(1)
\end{align*}
Note that $u_k \psi_k$ is compactly supported in $H$.
So, applying the classical Sobolev inequality and \eqref{eqn: trace Sobolev} to the first and second term respectively and using the continuity of $T\mapsto \Phi_H(T)$, we find
\begin{align*}
    \int_H |\nabla u_k|^p\, dx
& \geq m_k^p S_{n,p}^p  + \ell_k^p \Phi_H(T/\ell_k)^p + o_k(1) = S_{n,p}^p m_*^p + \ell_*^p \Phi_H(T/\ell_*)^p + o_k(1).
\end{align*}
Keeping in mind that $S_{n,p}=\Phi_H(0)$ and that $\int_H |\nabla u_k|^p \,dx = \Phi_H(T)^p +o_k(1)$, this
 violates Theorem~\ref{thm: strict energy gap}.

Next, suppose $m_*=1$ and thus $\ell_*=0$. Notice that $\ell_k>0$ for each $k$, otherwise the trace constraint is violated. In this case, arguing similarly but now applying the Sobolev inequality to the first term and  \eqref{eqn: trace Sobolev}  followed by \eqref{eqn: asymptotic} to the second,  we find 
\begin{align*}
\int_H |\nabla u_k|^p\, dx &\geq \int_H |\nabla(u_k \psi_{k})|^p \,dx + \int_H |\nabla (u_k(1-\psi_k))|^p \,dx \\
& \geq m_k^pS_{n,p}^p  + \ell_k^p \Phi_H(T/\ell_k)^p + o_k(1)\geq m_*^p S_{n,p}^p  + \Big(\frac{1}{p^\sharp}\frac{T^{p^\sharp}}{\ell_k^{p^\sharp -1}}\Big)^p+o_k(1)\to +\infty,
\end{align*}
a contradiction. We conclude that $d_k$ is bounded and thus vanishing does not occur.\\

{\it Dichotomy does not occur.}
Suppose by way of contradiction that the dichotomy alternative holds with splitting proportion $\lambda\in (0,1)$. Take a sequence $\varepsilon_k\to 0$. By Lemma~\ref{lem: CC1} and a diagonal argument, up to a subsequence, we may find $R_k \to \infty$ and $\{x_k\}$ such that the measures $\nu_k^1 = \nu_k \llcorner B_{R_k}(x_k)$ and $ \nu_k^2 = \nu_k \llcorner (\overline{H} \setminus B_{8R_k}(x_k))$ satisfy
\begin{equation}
    \label{eqn: dichotomy outcome}
\limsup_{k \to \infty}\big\{|\nu_k^1(\overline{H}) -\lambda| + |\nu_k^2(\overline{H}) -(1-\lambda)|\big\}=0\,.
\end{equation}

Now, take a smooth cutoff function $\varphi: \R^n \to [0,1]$ with $\varphi= 1$ on $B_2$ and $\varphi=0$ on $\R^n\setminus B_3$. Similarly let $\eta: \R^n \to [0,1]$ have $\eta = 0$ in $B_6$ and $\eta = 1$ on $\R^n\setminus B_7$. 
Let $\varphi_k (x) = \varphi (\frac{x-x_k}{R_k})$ and $\eta_k(x) = \eta (\frac{x-x_k}{R_k})$. Then setting $\vv_1^{p^*}=\lambda$ and $\vv_2^{p^*}=1-\lambda$, from \eqref{eqn: dichotomy outcome} we have
\[
\int_H|u_k\varphi_k|^{p^*}\,dx = \vv_1^{p^*} +o_k(1), \qquad \int_H |u_k\eta_k|^{p^*}\,dx = \vv_2^{p^*} +o_k(1)
\]
By \eqref{eqn: dichotomy outcome} and Lemma~\ref{lem: small trace}, we also have $\int_{\partial H \cap B_{7R_k}(x_k)\setminus B_{2R_k}(x_k)} |u_k|^{p^\sharp} \, d\mathcal{H}^{n-1} \to 0$, so
in particular,
\[
T^{p^\sharp} =\int_{\partial H} |u_k \varphi_k|^{p^\sharp} \,d\mathcal{H}^{n-1} + \int_{\partial H} |u_k \eta_k|^{p^\sharp} \,d\mathcal{H}^{n-1} +o_k(1)\,.
\]
Up to passing to a further subsequence, there exist $\ttt_1, \ttt_2\geq 0$ with $\ttt_1^{p^\sharp} + \ttt_2^{p^\sharp} = T^{p^\sharp}$ such that 
\[
\int_{\partial H} |u_k \varphi_k|^{p^\sharp} \,d\mathcal{H}^{n-1}=  \ttt_1^{p^\sharp}  +o_k(1),\qquad \int_{\partial H} |u_k \eta_k|^{p^\sharp} \,d\mathcal{H}^{n-1}= \ttt_2^{p^\sharp}  +o_k(1)\,.
\]

Similarly, using $1\geq \varphi^p +\eta^p$ and $\nu_k(B_{8R_k}(x_k)\setminus B_{R_k}(x_k))\to0$, and applying H\"{o}lder's inequality and scaling as in the proof of Lemma~\ref{lem: small trace}, we obtain
\[
\int_H |\nabla u_k|^p \,dx \geq \int_H |\nabla (u_k \varphi_k)|^p\,dx  +\int_H |\nabla (u_k\eta_k)|^p\,dx + o_k(1)\,.
\]
So, applying \eqref{eqn: trace Sobolev} to  $u_k\varphi_k$ and $u_k\eta_k$ separately and using the continuity of the mapping $T\mapsto \Phi_H(T)$, we have 
\begin{align*}
  \int_H |\nabla (u_k \varphi_k)|^p\,dx  +\int_H |\nabla (u_k\eta_k)|^p\,dx  \geq \vv_1^p \Phi_H\Big(\mfrac{\ttt_1}{\vv_1}\Big)^p + \vv_2^p \Phi_H\Big(\mfrac{\ttt_2}{\vv_2}\Big)^p + o_k(1).
\end{align*}
Since on the other hand, the left-hand side is bounded above by $\int_H |\nabla u_k|^p \,dx= \Phi_H(T)^p+o_k(1)$, we reach a contradiction to Theorem~\ref{thm: strict energy gap} for sufficiently large $k$. Thus the dichotomy alternative cannot occur.\\

So, the concentration alternative occurs in Lemma~\ref{lem: CC1}. In particular, we can find $\{x_k\}\subset \R^n$ and $R_0>0$ such that $\int_{B_{R_0}(x_k)}|u_k|^{p^*} >\frac{1}{2}$. Thanks to \eqref{eqn: normalize}, this means $B_1(d_k e_1)\cap B_{R_0}(x_k)$ is nonempty. Since $|d_k|$ is bounded uniformly in $k$, the same is true for $|x_k|$. So, the concentration alternative guarantees that for any $\varepsilon>0$, there exists $R>0$ such that
\[
\int_{B_R(0)} d \nu_k \geq 1-\varepsilon.
\]
So, up to a subsequence, $\nu_k \overset{*}\rightharpoonup \nu$ for a measure $\nu$ on $\overline{H}$ with $\nu(\overline{H})=1$. Applying the argument of Lemma~\ref{lem: small trace} but now taking $\psi:\R^n\to\R$ to be a cutoff with $\psi =0$ in $B_1$ and $\psi = 1$ on $\R^n\setminus B_2$, we also see that the measures $\tau_k = |u_k|^{p^\sharp} \,\mathcal{H}^{n-1}\llcorner \partial H$ have $\tau_k \overset{*}\rightharpoonup \tau$ for a measure $\tau$ on $\partial H$ with $\tau(\partial H)=T^{p^\sharp}.$\\

Now, let $\mu_k = |\nabla u_k|^p \mathcal{L}^n\llcorner H$. Up to a further subsequence, $\mu_k \overset{*}\rightharpoonup \mu$ for measure $\mu$ on $\overline{H}$. We apply Lemma~\ref{lem: CC2}. Up to a subsequence, $u_k \rightharpoonup u$ in $\dot{W}^{1,p}(H)$, $L^{p^*}(H)$, and $L^{p^\sharp}(\partial H)$. Then, since $\delta_T(u_k)\to0,$
\begin{align}\label{eqn: final part}
    \Phi_H(T)^p = \lim_{k\to \infty} \mu_k(\overline{H}) \geq  \mu(\overline{H})& \geq \| u\|_{L^{p^*}(H)}^{p}\Phi_H\Big(\mfrac{\|u\|_{L^{p^\sharp}(\partial H)}}{\| u\|_{L^{p^*}(H)}}\Big)^p + \sum_{i \in I} \vv_i^p \Phi_H\Big(\mfrac{\ttt_i}{\vv_i}\Big)^p\,\\
   &\geq \| u\|_{L^{p^*}(H)}^{p}\Phi_H\Big(\mfrac{\|u\|_{L^{p^\sharp}(\partial H)}}{\| u\|_{L^{p^*}(H)}}\Big)^p +\big( \sum_{i \in I} \vv_i^{p^*}\big)^{p/p^*}\Phi_H\bigg(\mfrac{(\sum_{i\in I}\ttt_i^{p^\sharp})^{1/p^\sharp}}{(\sum_{i \in I} \vv_i^{p^*})^{1/p^*}}\bigg)^p .
   \nonumber
\end{align}
In the last inequality, we use that Theorem~\ref{thm: strict energy gap} passes to countable sums at the expense of losing the strict inequality; the proof is a basic analysis exercise.
Finally, we apply Theorem~\ref{thm: strict energy gap} to bound the right-hand side below by $\Phi_H(T)^p$. The inequality must be strict (a contradiction) unless either $\| u\|_{L^{p^*}(H)}$ or $\sum_{i \in I} \vv_i^{p^*}$ is zero. 
If $\| u\|_{L^{p^*}(H)}=0$, then $\sum_{i \in I} \vv_i^{p^*}=1$.
Since the normalization \eqref{eqn: normalize} guarantees that $\vv_i^{p^*} \leq 1/2$ for each $i\in I$, then instead peeling off the first term in the sum on the right-hand side of \eqref{eqn: final part} and again applying Theorem~\ref{thm: strict energy gap} yields a contradiction. 
So, $\vv_i=0$ for each $i \in I$ and $\nu  = |u|^{p^*}\,\mathcal{L}^n$.
From Lemma~\ref{lem: CC2}, this means the index set $I$ is empty and thus $\tau = |u|^{p^\sharp} \mathcal{H}^{n-1} \llcorner \partial H$. 

This means $\|u_k\|_{L^{p^*}(H)} \to \|u\|_{L^{p^*}(H)}$ and $\|u_k\|_{L^{p^\sharp}(\partial H)} \to \|u\|_{L^{p^\sharp}(\partial H)}$, and so the weak convergence upgrades to strong convergence: $u_k \to u$ in $L^{p^*}(H)$ and $L^{p^\sharp}(\partial H).$ 
In particular, $u \in \compset$ and $\int_H |\nabla u|^p \geq \Phi_H(T)^p$. 
Combining this with lower semicontinuity and the fact that $\int_H |\nabla u_k|^p \to \Phi_H(T)^p$,  we have $\|\nabla u_k\|_{L^p(H)}\to \|\nabla u\|_{L^p(H)} =\Phi_H(T)$. This means (a) $u \in \mathcal{M}_T$, and (b) again the weak convergence upgrades to $\nabla u_k \to \nabla u$ in $L^p(H)$.  Scaling back, this means that for the original sequence, $d_T(\hat{u}_k)\to 0.$ We have passed to various subsequences, but since each subsequence has a further subsequence to which the argument applies, we obtain the conclusion for the entire sequence.
\end{proof}

Finally we show how Theorem~\ref{thm: qual stability} combined with the local stability result of \cite{FLZ} implies Corollary~\ref{cor: quant stability}.
\begin{proof}[Proof of Corollary~\ref{cor: quant stability}]
    Observe that for any $u \in \compset,$
    \[
d_T(u)^2\leq 2\|\nabla u\|_{L^2(H)}^2 +2 \|\nabla U_T\|_{L^2(H)}^2\leq 4 \|\nabla u\|_{L^2(H)}^2.
    \]
    So, for any fixed number $\delta_0\in (0,1)$, the desired inequality holds with constant $\alpha_T' = \frac{\delta_0}{4}$ among those functions $u\in \compset$ with $\delta_T(u) \geq \delta_0\| \nabla u\|_{L^2(H)}^2$. 
It thus suffices to prove the inequality when $\delta_T(u) \leq \delta_0 \| \nabla u\|_{L^2(H)}^2$, or equivalently after rearranging terms, when $\delta_T(u) \leq \frac{\delta_0}{1-\delta_0}\Phi_H(T)^2$ for a small enough fixed $\delta_0$ of our choosing. 

Let $\varepsilon>0$ be chosen so that, in the main result of \cite{FLZ}, we have $\delta_T(u) \geq \frac{\alpha_T}{2} d_T(u)^2$ provided $d_T(u)^2 \leq \varepsilon$. 
By Theorem~\ref{thm: qual stability}, there exists $\delta_0>0$ such that if $\delta_T(u) \leq \frac{\delta_0}{1-\delta_0}\Phi_H(T)^2$, then $d_T(u)^2 \leq \varepsilon$. 
Thus the desired stability inequality holds with $\alpha_T' = \frac{\alpha_T}{2}$ for $u\in \compset$ with $\delta_T(u) \leq \delta_0 \| \nabla u\|_{L^2(H)}^2$, and therefore for all $u \in \compset$ with $\alpha_T' = \min\{\frac{\alpha_T}{2}, \frac{\delta_0}{4}\}.$ 
\end{proof}

\bibliographystyle{alpha}
\bibliography{references}

\newcommand{\etalchar}[1]{$^{#1}$}
\begin{thebibliography}{CFMP09}

\bibitem[Aub76a]{AubinYamabe}
T.~Aubin.
\newblock \'{E}quations diff\'{e}rentielles non lin\'{e}aires et probl\`eme de
  {Y}amabe concernant la courbure scalaire.
\newblock {\em J. Math. Pures Appl. (9)}, 55(3):269--296, 1976.

\bibitem[Aub76b]{aubin1976}
T.~Aubin.
\newblock Probl\`{e}mes isop\'{e}rim\'{e}triques et espaces de {S}obolev.
\newblock {\em J. Differential Geom.}, 11(4):573--598, 1976.

\bibitem[BE91]{BianchiEgnell91}
G.~Bianchi and H.~Egnell.
\newblock A note on the {S}obolev inequality.
\newblock {\em J. Funct. Anal.}, 100(1):18--24, 1991.

\bibitem[Bec93]{Beckner93}
W.~Beckner.
\newblock Sharp {S}obolev inequalities on the sphere and the
  {M}oser-{T}rudinger inequality.
\newblock {\em Ann. of Math. (2)}, 138(1):213--242, 1993.

\bibitem[Bre91]{Brenier91}
Y.~Brenier.
\newblock Polar factorization and monotone rearrangement of vector-valued
  functions.
\newblock {\em Comm. Pure Appl. Math.}, 44(4):375--417, 1991.

\bibitem[CENV04]{CENV2004}
D.~Cordero-Erausquin, B.~Nazaret, and C.~Villani.
\newblock A mass-transportation approach to sharp {S}obolev and
  {G}agliardo-{N}irenberg inequalities.
\newblock {\em Adv. Math.}, 182(2):307--332, 2004.

\bibitem[CFMP09]{ciafusmag07}
A.~Cianchi, N.~Fusco, F.~Maggi, and A.~Pratelli.
\newblock The sharp {S}obolev inequality in quantitative form.
\newblock {\em J. Eur. Math. Soc.}, 11(5):1105--1139, 2009.

\bibitem[CL90]{CLcompetingsym}
Eric~A. Carlen and Michael Loss.
\newblock Extremals of functionals with competing symmetries.
\newblock {\em J. Funct. Anal.}, 88(2):437--456, 1990.

\bibitem[CL94]{carlenloss}
E.~A. Carlen and M.~Loss.
\newblock On the minimization of symmetric functionals.
\newblock {\em Rev. Math. Phys.}, 6(5A):1011--1032, 1994.
\newblock Special issue dedicated to Elliott H. Lieb.

\bibitem[DEF{\etalchar{+}}25]{DEFFL}
J.~Dolbeault, M.~J. Esteban, A.~Figalli, R.~L. Frank, and M.~Loss.
\newblock Sharp stability for {S}obolev and log-{S}obolev inequalities, with
  optimal dimensional dependence.
\newblock {\em Camb. J. Math.}, 13(2):359--430, 2025.

\bibitem[Esc88]{Escobar1988}
J.~F. Escobar.
\newblock Sharp constant in a {S}obolev trace inequality.
\newblock {\em Indiana Univ. Math. J.}, 37(3):687--698, 1988.

\bibitem[Esc92a]{Escobar92}
J.~F. Escobar.
\newblock Conformal deformation of a {R}iemannian metric to a scalar flat
  metric with constant mean curvature on the boundary.
\newblock {\em Ann. of Math. (2)}, 136(1):1--50, 1992.

\bibitem[Esc92b]{Escobar92b}
J.~F. Escobar.
\newblock The {Y}amabe problem on manifolds with boundary.
\newblock {\em J. Differential Geom.}, 35(1):21--84, 1992.

\bibitem[FLZ26]{FLZ}
S.~Fan, G.-D. Li, and J.J. Zhang.
\newblock Stability of the {S}obolev--{E}scobar bridge inequality.
\newblock {\em arXiv:2604.12677}, 2026.

\bibitem[FMP10]{FMP}
A.~Figalli, F.~Maggi, and A.~Pratelli.
\newblock A mass transportation approach to quantitative isoperimetric
  inequalities.
\newblock {\em Invent. Math.}, 182(1):167--211, 2010.

\bibitem[FMP13]{figmagpraa}
A.~Figalli, F.~Maggi, and A.~Pratelli.
\newblock Sharp stability theorems for the anisotropic {S}obolev and
  log-{S}obolev inequalities on functions of bounded variation.
\newblock {\em Adv. Math.}, 242:80--101, 2013.

\bibitem[FN19]{FN}
A.~Figalli and R.~Neumayer.
\newblock Gradient stability for the {S}obolev inequality: the case {$p\geq
  2$}.
\newblock {\em J. Eur. Math. Soc. (JEMS)}, 21(2):319--354, 2019.

\bibitem[FZ22]{FZ}
A.~Figalli and Y.~R.-Y. Zhang.
\newblock Sharp gradient stability for the {S}obolev inequality.
\newblock {\em Duke Math. J.}, 171(12):2407--2459, 2022.

\bibitem[Ho22]{Ho22}
P.~T. Ho.
\newblock A note on the {S}obolev trace inequality.
\newblock {\em Proc. Amer. Math. Soc.}, 150(3):1257--1267, 2022.

\bibitem[Lio84]{CCLocallyCompact}
P.-L. Lions.
\newblock The concentration-compactness principle in the calculus of
  variations. {T}he locally compact case. {I}.
\newblock {\em Ann. Inst. H. Poincar\'e{} Anal. Non Lin\'eaire}, 1(2):109--145,
  1984.

\bibitem[Lio85]{lions1985}
P.-L. Lions.
\newblock The concentration-compactness principle in the calculus of
  variations. {T}he limit case. {I}.
\newblock {\em Rev. Mat. Iberoamericana}, 1(1):145--201, 1985.

\bibitem[Mag23]{MaggiOTBook}
F.~Maggi.
\newblock {\em Optimal mass transport on {E}uclidean spaces}, volume 207 of
  {\em Cambridge Studies in Advanced Mathematics}.
\newblock Cambridge University Press, Cambridge, 2023.

\bibitem[McC97]{McCann97}
R.J. McCann.
\newblock A convexity principle for interacting gases.
\newblock {\em Adv. Math.}, 128(1):153--179, 1997.

\bibitem[MN17]{MN}
F.~Maggi and R.~Neumayer.
\newblock A bridge between {S}obolev and {E}scobar inequalities and beyond.
\newblock {\em J. Funct. Anal.}, 273(6):2070--2106, 2017.

\bibitem[MNT23]{MNT}
F.~Maggi, R.~Neumayer, and I.~Tomasetti.
\newblock Rigidity theorems for best {S}obolev inequalities.
\newblock {\em Adv. Math.}, 434:Paper No. 109330, 43, 2023.

\bibitem[MV05]{maggivillaniJGA}
F.~Maggi and C.~Villani.
\newblock Balls have the worst best {S}obolev inequalities.
\newblock {\em J. Geom. Anal.}, 15(1):83--121, 2005.

\bibitem[Naz06]{Nazaret2006}
B.~Nazaret.
\newblock Best constant in {S}obolev trace inequalities on the half-space.
\newblock {\em Nonlinear Anal.}, 65(10):1977--1985, 2006.

\bibitem[Neu20]{Neumayer}
R.~Neumayer.
\newblock A note on strong-form stability for the {S}obolev inequality.
\newblock {\em Calc. Var. Partial Differential Equations}, 59(1):Paper No. 25,
  8, 2020.

\bibitem[Sch84]{SchYamabe}
R.~Schoen.
\newblock Conformal deformation of a {R}iemannian metric to constant scalar
  curvature.
\newblock {\em J. Differential Geom.}, 20(2):479--495, 1984.

\bibitem[Str08]{Struwe}
M.~Struwe.
\newblock {\em Variational methods}, volume~34 of {\em Ergebnisse der
  Mathematik und ihrer Grenzgebiete. 3. Folge. A Series of Modern Surveys in
  Mathematics [Results in Mathematics and Related Areas. 3rd Series. A Series
  of Modern Surveys in Mathematics]}.
\newblock Springer-Verlag, Berlin, fourth edition, 2008.
\newblock Applications to nonlinear partial differential equations and
  Hamiltonian systems.

\bibitem[Tal76]{talenti1976best}
G.~Talenti.
\newblock Best constant in {S}obolev inequality.
\newblock {\em Ann. Mat. Pura Appl. (4)}, 110:353--372, 1976.

\bibitem[Tru68]{Trudinger}
N.~S. Trudinger.
\newblock Remarks concerning the conformal deformation of {R}iemannian
  structures on compact manifolds.
\newblock {\em Ann. Scuola Norm. Sup. Pisa Cl. Sci. (3)}, 22:265--274, 1968.

\bibitem[Vil03]{VillaniBook}
C.~Villani.
\newblock {\em Topics in optimal transportation}, volume~58 of {\em Graduate
  Studies in Mathematics}.
\newblock American Mathematical Society, Providence, RI, 2003.

\bibitem[Yam60]{OGYamabe}
H.~Yamabe.
\newblock On a deformation of {R}iemannian structures on compact manifolds.
\newblock {\em Osaka Math. J.}, 12:21--37, 1960.

\end{thebibliography}

\end{document}